\documentclass{article}
\usepackage{amsmath}
\usepackage{amssymb}
\usepackage{doublespace}

\newtheorem{theorem}{Theorem}

\newtheorem{proposition}{Proposition}
\newtheorem{lemma}{Lemma}
\newtheorem{remark}{Remark}
\newtheorem{definition}{Definition}
\newtheorem{example}{Example}
\newtheorem{notation}{Notation}
\newtheorem{conjecture}{Conjecture}
\numberwithin{equation}{section}
\numberwithin{theorem}{section}
\numberwithin{lemma}{section}
\numberwithin{proposition}{section}
\numberwithin{definition}{section}
\numberwithin{remark}{section}
\newenvironment{proof}[1][Proof]{\noindent\textbf{#1.} }{\ \rule{0.5em}{0.5em}}
\input{tcilatex}

\begin{document}

\title{Illumination by Tangent Lines}
\author{Alan Horwitz \\
25 Yearsley Mill Rd.\\
Penn State Brandywine\\
Media, PA 19063\\
USA\\
alh4@psu.edu}
\date{7/27/11}
\maketitle

\begin{abstract}
Let $f$ be a differentiable function on the real line, and let $P\in
G_{f}^{C}=$ all points not on the graph of $f$. We say that the \textit{%
illumination index} of $P$, denoted by $I_{f}(P)$, is $k$ if there are $k$
distinct tangents to the graph of $f$ which pass through $P$. In section 2
we prove results about the illumination index of $f$ with $f\,^{\prime
\prime }(x)\geq 0$ on $\Re $. In particular, suppose that $y=L_{1}(x)$ and $%
y=L_{2}(x)$ are distinct oblique asymptotes of $f$ and let $P=(s,t)\in
G_{f}^{C}$. If $\max \left( L_{1}(s),L_{2}(s)\right) <t<f(s)$, then $%
I_{f}(P)=2$. If $L_{1}(s)\neq L_{2}(s)$\textbf{\ }$\ $and $\min \left(
L_{1}(s),L_{2}(s)\right) <t\leq \max \left( L_{1}(s),L_{2}(s)\right) $, then 
$I_{f}(P)=1$.

Finally, if $t\leq \min \left( L_{1}(s),L_{2}(s)\right) $, then $I_{f}(P)=0$%
. We also show that any point below the graph of a convex rational function
or exponential polynomial must have illumination index equal to $2$. In
section 3 we also prove results about the illumination index of polynomials.


\end{abstract}

\section{Introduction\label{S1}}

Let $f$ be a differentiable function on the real line, $\Re $, and let $P$
be any point not on the graph of $f$. We say that the \textit{illumination
index} of $P$, denoted by $I_{f}(P)$, equals the non--negative integer, $k$,
if there are $k$ distinct tangents to the graph of $f$ which pass through $P$%
. We allow the possibility that $k=\infty $. In \cite{H} we proved some
results about $I_{f}(P)$ and also about illumination by odd order Taylor
Polynomials in general. In this paper we focus just on illumination by
tangent lines. In particular, in section 1 we prove several theorems about $%
I_{f}(P)$ for functions with a non--negative second derivative on $\Re $. An
example was given in \cite{H} where $f\,^{\prime \prime }(x)\geq 0$ on $\Re $%
, but where there are points below the graph of $f$ whose illumination index
equals $0$. In this paper we strengthen these results for functions, $f$,
with $f\,^{\prime \prime }(x)\geq 0$ on $\Re $. First, we prove(Theorem \ref%
{T2}) that if $f$ has oblique asymptotes $L_{1}$and $L_{2}$, then the
illumination index equals $2$ for any point below the graph of $f$, but
above both $L_{1}$ and $L_{2}$. For points lying between $L_{1}$ and $L_{2}$%
, the illumination index equals $1$, while for any point below both $L_{1}$
and $L_{2}$, the illumination index equals $0$. Similar results(Theorem \ref%
{T3}) are proven when $f$ has one oblique asymptote. In (\cite{H}) we proved
that if the second derivative of $f$ is bounded below by a positive number
on the entire real line, and if $P$ is any point below the graph of $f$,
then the illumination index of $P$ equals $2$. We strengthen this result in
Theorem \ref{T4} by proving that if $\lim\limits_{\left\vert x\right\vert
\rightarrow \infty }\left( xf\,^{\prime \prime }(x)\right) \neq 0$, then the
illumination index of any point below the graph of $f$ equals $2$. We also
show(Propositions \ref{P1} and \ref{P2}) that any point below the graph of a
convex rational function or exponential polynomial must have illumination
index equal to $2$. Finally in section 3 we prove several results about the
illumination index for polynomials.

\begin{notation}
$\Re =$ real numbers. Given any function, $y=f(x)$ defined on $\Re $, we let

\begin{equation*}
G_{f}=(x,f(x)):x\in \Re ,
\end{equation*}

the graph of $f$, and 
\begin{equation*}
G_{f}^{C}=(x,y):x\in \Re ,y\neq f(x),
\end{equation*}

all points in the $xy$ plane \textbf{not} on the graph of $f$.

\begin{equation*}
y=T_{c}(x)=f(c)+f\,^{\prime }(c)(x-c)
\end{equation*}

denotes the tangent line to $f$ at $(c,f(c))$.

For $s\in \Re $, we let 
\begin{equation*}
I_{1}=(-\infty ,s),I_{2}=(s,\infty ).
\end{equation*}
\end{notation}

\begin{definition}
\label{D1}Let $f$ be a differentiable function on the real line, and let $%
P\in G_{f}^{C}$. We say that the \textit{illumination index} of $P$, denoted
by $I_{f}(P)$, equals the non--negative integer $k$, if there are $k$
distinct tangents to the graph of $f$ which pass through $P$.
\end{definition}

\begin{remark}
In the definition above, one could allow for points, $P\in G_{f}$. However,
we prefer to just define $I_{f}(P)$ for $P\in G_{f}^{C}$. Also, if there are 
$k$ distinct tangents to the graph of $f$ which pass through $P,$ and if at
least one of the tangent lines is tangent to the graph of $f$ at more than
one point, we still count the illumination index\textit{\ as }$k$. One
could, of course, define $I_{f}(P)$ so as to count the number of points at
which $T$ is tangent to the graph of $f$ .
\end{remark}

Before stating and proving our main results, it is useful to define the
following function: If $f$ is a differentiable function on $\Re $ and $s\in
\Re $, let 
\begin{equation*}
g_{s}(c)=f(c)+(s-c)f\,^{\prime }(c)=T_{c}(s).
\end{equation*}%
Then 
\begin{equation}
T_{c_{0}}(s)=t\iff g_{s}(c_{0})=t.  \label{4}
\end{equation}%
That is, the tangent line at $\left( c_{0},f(c_{0})\right) $ passes thru $%
P=(s,t)$ if and only if $g_{s}(c_{0})=t$.

\begin{remark}
We find it convenient to use the notation $g_{s}(c)$ rather than $T_{c}(s)$
since we want to keep $s$ fixed while allowing $c$ to vary.
\end{remark}

\section{Functions with Non--negative Second Derivative\label{S2}}

In this section we prove some results about the illumination index for
functions, $f$, with $f\,^{\prime \prime }(x)\geq 0$ on $\Re $. We do not
assume continuity of $f\,^{\prime \prime }$, but the existence of $%
f\,^{\prime \prime }$ on $\Re $ implies that $g_{s}$ is differentiable on $%
\Re $. First we need the following result about \textit{multiple tangent
lines}, which is a tangent line which is tangent to the graph of $f$ at more
than one point.

\begin{lemma}
\label{mult}If $f\,^{\prime \prime }(x)\geq 0$ on $\Re $ and $f$ is not
linear on any subinterval of $\Re $, then $f$ has no multiple tangent lines.
\end{lemma}

\begin{proof}
Suppose that $f$ has a multiple tangent line, $T$, which is tangent at $%
(a,f(a))$ and at $(b,f(b))$ for some $a\neq b$. Then $\dfrac{p(b)-p(a)}{b-a}%
=p^{\prime }(a)=p^{\prime }(b)$. Since $f\,^{\prime \prime }(x)\geq 0$ on $%
\Re $, $p^{\prime }(a)\leq p^{\prime }(x)\leq p^{\prime }(b)$ for any $x\in
\lbrack a,b]$. Thus $p^{\prime }(x)$ is constant on $[a,b]$, which implies
that $f$ is linear on $[a,b]$.
\end{proof}

For functions, $f$, with $f\,^{\prime \prime }(x)\geq 0$ on $\Re $, part (i)
of the following lemma shows that $g_{s}$ has one local extremum, a local
maximum when $c=s$. Part (ii) shows that if $c_{i}<c_{j}$ are any two roots
of $g_{s}-t$, then there are two possibilities: Either the tangents to $f$
at $\left( c_{i},f\left( c_{i}\right) \right) $ and at $\left( c_{j},f\left(
c_{j}\right) \right) $ are distinct, or $f$ is linear on the closed interval 
$[c_{i},c_{j}]$.

\begin{lemma}
\label{L1}(i) If $f\,^{\prime \prime }(x)\geq 0$ on $\Re $, then for any
given $s\in \Re ,g_{s}(c)$ is non--decreasing on $I_{1}$ and non--increasing
on $I_{2}$.

(ii) For given $t\in \Re $, there are two possibilities for the number of
solutions of the equation $g_{s}(c)=t$ in $I_{j},j=1,2$.

(A) $g_{s}(c)=t$ has at most one solution in $I_{j}$, or

(B) $g_{s}(c)=t$ for all $c$ in some interval, $I$, contained in $I_{j}$. In
that case $f(x)=mx+b$ for all $x\in I$, which implies that $T_{c}(x)=f(x)$
for

$c,x\in I$. In addition, $g_{s}(c)=f(s)$ if $s\in I$.
\end{lemma}

\begin{proof}
(i) Since\textbf{\ }$g_{s}^{\prime }(c)=(s-c)f\,^{\prime \prime }(c)$, $%
g_{s}^{\prime }(c)$ is $\QDATOPD\{ \} {\geq 0,c<s}{\leq 0,c>s}$

(ii) By part (i), $g_{s}(c)=t$ has at most one solution in $I_{j}$, or $%
g_{s}(c)=t$ for all $c$ in some interval, $I$, contained in $I_{j}$. If $%
g_{s}(c)$ is constant on $I$, then $g_{s}^{\prime }(c)=0,c\in I$, which
implies that $(s-c)f\,^{\prime \prime }(c)=0,c\in I$. Thus $f\,^{\prime
\prime }(x)=0$ for $x\in I,x\neq s$, which implies that $f$ is linear on $I$%
. That implies that $T_{c}(x)=f(x)$ for $c,x\in I$. If $s\in I$, then $%
g_{s}(c)=T_{c}(s)=f(s)$.
\end{proof}

The following lemma was proved in \cite{H} with the assumption that $%
f\,^{\prime \prime }$ is continuous, non--negative, and has \textit{finitely
many }zeros in $\Re $. We have need for a somewhat stronger version here.

\begin{lemma}
\label{L2}Suppose that $f\,^{\prime \prime }(x)\geq 0$ on $\Re $. Then at
most two distinct tangent lines to $f$ can pass through any given point $P$
in the plane.
\end{lemma}

\begin{proof}
The details follow exactly as in the proof of (\cite{H}, Lemma 2) using the
following facts: Suppose that $T_{1}$ and $T_{2}$ are distinct tangent lines
which are tangent to $f$ at $\left( c_{1},f\left( c_{1}\right) \right) $ and 
$\left( c_{2},f\left( c_{2}\right) \right) $, respectively. Then $T_{1}$ and 
$T_{2}$ are not parallel and if $(u,v)=$ intersection point of $T_{1}$ and $%
T_{2}$, then $c_{1}<u<c_{2}$. We leave the rest of the details to the reader.
\end{proof}

\begin{definition}
\label{D2}A line with equation $y=L(x)$ is said to be an oblique asymptote
of $f$ if $\lim\limits_{x\rightarrow -\infty }(f(x)-L(x))=0$ and/or $%
\lim\limits_{x\rightarrow \infty }(f(x)-L(x))=0$.
\end{definition}

\begin{lemma}
\label{L3}Suppose that $f\,^{\prime \prime }(x)\geq 0$ on $\Re $, and let $%
s\in \Re $.
\end{lemma}

(i) If $\lim\limits_{x\rightarrow \infty }(f(x)-L(x))=0$ for some linear
function, $L$, then

$\lim\limits_{c\rightarrow \infty }g_{s}(c)=L(s)$

(ii) If $\lim\limits_{x\rightarrow \,-\infty }(f(x)-L(x))=0$ for some linear
function, $L$, then

$\lim\limits_{c\rightarrow -\infty }g_{s}(c)=L(s)$

\begin{proof}
We prove (i). Assume first that $L=0$--that is, $\lim\limits_{x\rightarrow
\infty }f(x)=0$. Choose any $h>0$ and partition $[s,\infty )$ into
infinitely many subintervals, $[x_{k-1},x_{k}]$, of constant width $h$.
Since $f$ is convex, 
\begin{equation*}
\dfrac{f(x_{k})-f(x_{k-1})}{h}\leq f\,^{\prime }(x_{k})\leq \dfrac{%
f(x_{k+1})-f(x_{k})}{h}.
\end{equation*}%
Since $f(x_{k})-f(x_{k-1})\rightarrow 0$ and $f(x_{k+1})-f(x_{k})\rightarrow
0$ as $k\rightarrow \infty $, $f\,^{\prime }(x_{k})\rightarrow 0$ as $%
k\rightarrow \infty $ by the Squeeze Theorem. Since $h$ is arbitrary, that
proves that 
\begin{equation}
\lim\limits_{x\rightarrow \infty }f\,^{\prime }(x)=0.  \label{3}
\end{equation}%
Now $\dfrac{d}{dc}\left( f(c)-cf\,^{\prime }(c)\right) =-cf\,^{\prime \prime
}(c)\leq 0$ for $c>0$, which implies that $f(c)-cf\,^{\prime }(c)$ is
non--increasing for $c>0$. Since $f$ is convex and $\lim\limits_{x%
\rightarrow \infty }f(x)=0$, $f$ must be eventually positive and
non--increasing, which implies that $f\,^{\prime }(c)\leq 0$ for large $c$.
Thus $f(c)-cf\,^{\prime }(c)$ is eventually positive and non--increasing,
which implies that $f(c)-cf\,^{\prime }(c)$ is bounded and monotonic on $%
[0,\infty )$. Using the integral form of Taylor's Remainder formula, we have 
$f(s)-T_{c}(s)=\dint\limits_{s}^{c}(t-s)f\,^{\prime \prime }(t)dt$, which
implies that 
\begin{equation}
f(s)-g_{s}(c)=\dint\limits_{s}^{c}(t-s)f\,^{\prime \prime }(t)dt.  \label{1}
\end{equation}%
Let $s=0$ and use integration by parts with $u=t$ and $dv=f\,^{\prime \prime
}(t)dt$ to obtain $\dint\limits_{0}^{c}tf\,^{\prime \prime }(t)dt=\left[
tf\,^{\prime }(t)\right] _{0}^{c}-\dint\limits_{0}^{c}f\,^{\prime
}(t)dt=cf\,^{\prime }(c)-f(c)+f(0)$. Since $f(c)-cf\,^{\prime }(c)+f(0)$ is
bounded and monotonic on $[0,\infty )$, the improper integral $%
\dint\limits_{0}^{\infty }tf\,^{\prime \prime }(t)dt$ converges. Let $%
G(u)=L\left\{ tf\,^{\prime \prime }(t)\right\} (u)=\dint\limits_{0}^{\infty
}e^{-ut}tf\,^{\prime \prime }(t)dt$, where $L$ denotes the Laplace
Transform. Then $G(0)=\dint\limits_{0}^{\infty }tf\,^{\prime \prime }(t)dt$.
Using well known formulas for the Laplace Transform, with $F(s)=L(f)$, $%
L\left( tf\,^{\prime \prime }(t)\right) =-\dfrac{d}{du}L\left( f\,^{\prime
\prime }(t)\right) =-\dfrac{d}{du}\left( u^{2}L(f)-uf(0)-f\,^{\prime
}(0)\right) =-u^{2}F^{\prime }(s)-2uF(s)-f(0)$, which implies that $%
G(0)=f(0) $. Hence 
\begin{equation}
\dint\limits_{0}^{\infty }tf\,^{\prime \prime }(t)dt=f(0).  \label{2}
\end{equation}%
Since $\dint\limits_{0}^{c}tf\,^{\prime \prime }(t)dt=f(0)-g_{0}(c)$ by (\ref%
{1}), $\lim\limits_{c\rightarrow \infty }\left( f(0)-g_{0}(c)\right) =f(0)$
by (\ref{2}), which implies that $\lim\limits_{c\rightarrow \infty
}g_{0}(c)=0$. Now 
\begin{equation*}
\lim\limits_{c\rightarrow \infty }g_{s}(c)=s\lim\limits_{c\rightarrow \infty
}f\,^{\prime }(c)+\lim\limits_{c\rightarrow \infty }\left( f(c)-cf\,^{\prime
}(c)\right) =s(0)+\lim\limits_{c\rightarrow \infty }g_{0}(c)=0.
\end{equation*}%
That proves (i) when $L(x)=0$. Now assume that $\lim\limits_{x\rightarrow
\infty }(f(x)-L(x))=0$ and let $w(x)=f(x)-L(x)$. Then $\lim\limits_{x%
\rightarrow \infty }w(x)=0$ and $\bar{T}_{c}=T_{c}-L=$ tangent line to $w$
at $(c,w(c))$. By what we just proved, $\lim\limits_{c\rightarrow \infty }%
\bar{T}_{c}(s)=0$, which implies that $\lim\limits_{c\rightarrow \infty
}\left( T_{c}(s)-L(s)\right) =0$.
\end{proof}

We now prove some theorems about the illumination index of functions convex
on the real line. For any convex function, $f$, it is trivial that if $%
P=(s,t)$ lies above the graph of $f$, then $I_{f}(P)=0$. Thus we do not
bother stating that case in any of the theorems below.

\begin{theorem}
\label{T2}Suppose that $f\,^{\prime \prime }(x)\geq 0$ on $\Re $ and that $%
y=L_{1}(x)$ and $y=L_{2}(x)$ are distinct oblique asymptotes of $f$. Let $%
P=(s,t)\in G_{f}^{C}$ be given.

(i) If $\max \left( L_{1}(s),L_{2}(s)\right) <t<f(s)$, then $I_{f}(P)=2$.

(ii) If $L_{1}(s)\neq L_{2}(s)$\textbf{\ }$\ $and $\min \left(
L_{1}(s),L_{2}(s)\right) <t\leq \max \left( L_{1}(s),L_{2}(s)\right) $, then 
$I_{f}(P)=1$.

(iii) If $t\leq \min \left( L_{1}(s),L_{2}(s)\right) $, then $I_{f}(P)=0$.
\end{theorem}

\begin{proof}
Without loss of generality we can assume that $\lim\limits_{x\rightarrow
\,-\infty }(f(x)-L_{1}(x))=0$ and $\lim\limits_{x\rightarrow \infty
}(f(x)-L_{2}(x))=0$. Then by Lemma \ref{L3}, $\lim\limits_{c\rightarrow
\,-\infty }g_{s}(c)=L_{1}(s)$ and $\lim\limits_{c\rightarrow \infty
}g_{s}(c)=L_{2}(s)$. We prove the theorem for the case when $L_{1}(s)\leq
L_{2}(s)$, the proof when $L_{2}(s)\leq L_{1}(s)$ being similar. Thus we
have 
\begin{eqnarray*}
\lim\limits_{c\rightarrow \,-\infty }g_{s}(c) &=&\min \left(
L_{1}(s),L_{2}(s)\right) =L_{1}(s), \\
\lim\limits_{c\rightarrow \infty }g_{s}(c) &=&\max \left(
L_{1}(s),L_{2}(s)\right) =L_{2}(s).
\end{eqnarray*}

To prove (i): $\lim\limits_{c\rightarrow \,-\infty
}g_{s}(c)=L_{1}(s)<t,g_{s}(s)=f(s)>t$, and $\lim\limits_{c\rightarrow \infty
}g_{s}(c)=L_{2}(s)<t$. That implies that $g_{s}-t$ has at least two real
roots, $c_{1}\in I_{1}$ and $c_{2}\in I_{2}$, by the Intermediate Value
Theorem. Note that $g_{s}(s)=f(s)\neq t$, so that $c=s$ is not a root of $%
g_{s}-t$. Either $c_{1}$ is the only root of $g_{s}-t$ in $I_{1}$, or $%
g_{s}(c)=t$ for all $c$ in some interval, $I$, contained in $I_{1}$ by Lemma %
\ref{L1}(ii). In the latter case, $T_{c}(x)=f(x)$ for all $c,x\in I$, so
that there is only one tangent line for all $c\in I$. In either case, that
yields one tangent line from $I_{1}$ which passes thru $P$. The same holds
for $I_{2}$ by Lemma \ref{L1}(ii).Thus there are precisely two distinct
tangent lines to $f$ which pass thru $P$, which implies that $I_{f}(P)=2$.

To prove (ii): It follows easily, as in the proof of Lemma \ref{L3}, that $%
L_{2}(s)\leq f(s)$, which implies that $t<f(s)$ since $(s,t)\in G_{f}^{C}$.
Thus $g_{s}(s)\neq t$, so again $c=s$ is not a root of $g_{s}-t$. Note also
that $g_{s}(c)\neq t$ for any $c\in I_{2}$ since $g_{s}$ is non--increasing
on $(s,\infty )$, $t\leq L_{2}(s)$, and $\lim\limits_{c\rightarrow \infty
}g_{s}(c)=L_{2}(s)$. Since $\lim\limits_{c\rightarrow \,-\infty
}g_{s}(c)=L_{1}(s)<t$\ and $g_{s}(s)=f(s)>t$, $g_{s}-t$ has at least one
real root, $c_{0}\in I_{1}$. Either $c_{0}$ is the only root of $g_{s}-t$ in 
$I_{1}$, or $g_{s}(c)=t$ for all $c$ in some interval, $I$, contained in $%
I_{1}$ by Lemma \ref{L1}(ii). In the latter case, $T_{c}(x)=f(x)$ for all $%
c,x\in I$, so that there is only one tangent line for all $c\in I$. In
either case, that yields one tangent line from $I_{1}$ which passes thru $P$%
, which implies that $I_{f}(P)=1$.

To prove (iii): If $t<L_{1}(s)$, then it follows easily that $g_{s}(c)=t$
has no solution. If $t=L_{1}(s)$ and $g_{s}(c)=t$, then $g_{s}(c)=t$ for all 
$c\in I=(-\infty ,k)$ for some $k<s$. Arguing as in the proof of Lemma \ref%
{L1}(ii), it follows easily that $f(x)=L_{1}(x)$ for all $x\in I$ , which
implies that $g_{s}(c)=f(s)$ and thus $f(s)=t$, which contradicts the
assumption that $(s,t)\in G_{f}^{C}$. Hence $I_{f}(P)=0$.
\end{proof}

\begin{example}
Let $f(x)=x\tan ^{-1}x$. Then $f\,^{\prime \prime }(x)=\allowbreak \dfrac{2}{%
\left( 1+x^{2}\right) ^{2}}>0$ on $\Re $, and $y=\pm \dfrac{\pi }{2}x-1$are
distinct oblique asymptotes of $f$. Thus Theorem \ref{T2} applies with $%
L_{1}(x)=-\dfrac{\pi }{2}x-1$and $L_{2}(x)=\dfrac{\pi }{2}x-1$. 
\begin{eqnarray*}
\min \left( L_{1}(s),L_{2}(s)\right) &=&\left\{ 
\begin{array}{ll}
\dfrac{\pi }{2}s-1 & \text{if }s<0 \\ 
-\dfrac{\pi }{2}s-1 & \text{if }s\geq 0%
\end{array}%
\right. \\
\max \left( L_{1}(s),L_{2}(s)\right) &=&\left\{ 
\begin{array}{ll}
-\dfrac{\pi }{2}s-1 & \text{if }s<0 \\ 
\dfrac{\pi }{2}s-1 & \text{if }s\geq 0%
\end{array}%
\right.
\end{eqnarray*}

Let $P=(s,t)$. If $s<0$ and $-\dfrac{\pi }{2}s-1<t<s\tan ^{-1}s$, or $s\geq
0 $ and $\dfrac{\pi }{2}s-1<t<s\tan ^{-1}s$, then $I_{f}(P)=2$.

If $s<0$ and $\dfrac{\pi }{2}s-1<t\leq -\dfrac{\pi }{2}s-1$, or $s>0$ and $-%
\dfrac{\pi }{2}s-1<t\leq \dfrac{\pi }{2}s-1$, then $I_{f}(P)=1$.

Finally, if $s<0$ and $t\leq \dfrac{\pi }{2}s-1$, \ or $s\geq 0$ and $t\leq -%
\dfrac{\pi }{2}s-1$, then $I_{f}(P)=0$.
\end{example}

Before proving our next result, we need the following lemma.

\begin{lemma}
\label{L4}Suppose that $f\,^{\prime \prime }(x)\geq 0$ for $\left\vert
x\right\vert >b$, where $b$ is a positive real number. Let $%
g_{s}(c)=f(c)+(s-c)f\,^{\prime }(c)$ for given $s\in \Re $.

(i) If $\lim\limits_{x\rightarrow \,-\infty }\left( xf\,^{\prime \prime
}(x)\right) =A$, where $-\infty \leq A<0$, then $\lim\limits_{c\rightarrow
\,-\infty }g_{s}(c)=-\infty $

(ii) If $\lim\limits_{x\rightarrow \infty }\left( xf\,^{\prime \prime
}(x)\right) =A$, where $0<A\leq \infty $, then $\lim\limits_{c\rightarrow
\infty }g_{s}(c)=-\infty $
\end{lemma}

\begin{remark}
A weaker version of this lemma was given in (\cite{H}, Lemma 1) where it was
assumed that $f\,^{\prime \prime }(x)\geq m>0$ for $\left\vert x\right\vert
>b$, where $m$ and $b$ are positive real numbers.
\end{remark}

\begin{proof}
We prove (ii), the proof of (i) being similar. Let $\{x_{k}\}\subset
(b,\infty )$ be any sequence with $x_{k}\rightarrow \infty $. Suppose that $%
\lim\limits_{k\rightarrow \infty }f\,^{\prime \prime }\left( x_{k}\right)
=m>0$. Then $\lim\limits_{k\rightarrow \infty }\left[ (s-x_{k})f\,^{\prime
\prime }\left( x_{k}\right) \right] =-\infty \neq 0$. Second, suppose that $%
\lim\limits_{k\rightarrow \infty }f\,^{\prime \prime }\left( x_{k}\right) =0$%
. Then $\lim\limits_{k\rightarrow \infty }\left[ (s-x_{k})f\,^{\prime \prime
}\left( x_{k}\right) \right] =-\lim\limits_{k\rightarrow \infty }\left[
x_{k}f\,^{\prime \prime }\left( x_{k}\right) \right] \neq 0$ by (ii). Hence $%
\lim\limits_{x\rightarrow \infty }\left[ (s-x)f\,^{\prime \prime }(x)\right]
\neq 0$, which implies that $\lim\limits_{c\rightarrow \infty }g_{s}^{\prime
}(c)=\lim\limits_{c\rightarrow \infty }\left[ (s-c)f\,^{\prime \prime }(c)%
\right] \neq 0$. Since $f\,^{\prime \prime }(x)\geq 0$ for $x>b$, $g_{s}(c)$
is eventually decreasing. Since $\lim\limits_{c\rightarrow \infty
}g_{s}^{\prime }(c)\neq 0$, it follows that $\lim\limits_{c\rightarrow
\infty }g_{s}(c)=-\infty $.
\end{proof}

The following theorem is similar to Theorem \ref{T2} for the case when $f$
has only one oblique asymptote.

\begin{theorem}
\label{T3}Suppose that $f\,^{\prime \prime }(x)\geq 0$ on $\Re $ and that
one of the following two conditions holds, where $L$ is a linear function.

$\lim\limits_{x\rightarrow \,-\infty }\left( xf\,^{\prime \prime }(x)\right)
=A$, where $-\infty \leq A<0$ and $\lim\limits_{x\rightarrow \infty
}(f(x)-L(x))=0$, or $\lim\limits_{x\rightarrow \,-\infty }(f(x)-L(x))=0$ and 
$\lim\limits_{x\rightarrow \infty }\left( xf\,^{\prime \prime }(x)\right) =A$%
, where $0<A\leq \infty $. Let $P=(s,t)\in G_{f}^{C}$ be given.

(i) If $L(s)<t<f(s)$, then $I_{f}(P)=2$

(ii) If $t\leq L(s)$, then $I_{f}(P)=1$
\end{theorem}

\begin{proof}
We prove the case when $\lim\limits_{x\rightarrow \,-\infty }\left(
xf\,^{\prime \prime }(x)\right) =A$, $-\infty \leq A<0$ and $%
\lim\limits_{x\rightarrow \infty }(f(x)-L(x))=0$, the proof of the other
case being similar. By Lemma \ref{L4}, $\lim\limits_{c\rightarrow \,-\infty
}g_{s}(c)=-\infty $, and by Lemma \ref{L3}, $\lim\limits_{c\rightarrow
\infty }g_{s}(c)=L(s)$. If $L(s)<t<f(s)$, then $\lim\limits_{c\rightarrow
\,-\infty }g_{s}(c)<t$, $g_{s}(s)=f(s)>t$, and $\lim\limits_{c\rightarrow
\infty }g_{s}(c)<t$. That implies that $g_{s}-t$ has at least two real
roots, $c_{1}\in I_{1}$ and $c_{2}\in I_{2}$, by the Intermediate Value
Theorem. Arguing exactly as in the proof of Theorem \ref{T2}, part (i), it
follows that exactly two tangent lines pass thru $P$, which implies that $%
I_{f}(P)=2$. That proves (i). It follows easily, as in the proof of Lemma %
\ref{L3}, that $L(s)\leq f(s)$. If $t\leq L(s)$, then$\lim\limits_{c%
\rightarrow \,-\infty }g_{s}(c)<t$ and $g_{s}(s)=f(s)>t$ implies that $%
g_{s}-t$ has at least one real root, $c_{0}\in I_{1}$. Arguing exactly as in
the proof of Theorem \ref{T2}, part (ii), it follows that $I_{f}(P)=1$.
\end{proof}

\begin{remark}
It is possible to prove Theorem \ref{T3} with slightly weaker hypotheses.
However, we believe, but have not been able to prove, that the conclusion of
Theorem \ref{T3} holds with only the assumption that $f$ has one oblique
asymptote.
\end{remark}

\begin{example}
Let $f(x)=e^{x}$. Then$\lim\limits_{x\rightarrow \,-\infty }f(x)=0$ and $%
\lim\limits_{x\rightarrow \infty }\left( xf\,^{\prime \prime }(x)\right)
\neq 0$, so that Theorem \ref{T3} applies with $L(x)=0$. Hence,if $P=(s,t)$
with $0<t<e^{s}$, then $I_{f}(P)=2$. If $P=(s,t)$ with $t\leq 0$, then $%
I_{f}(P)=1$.
\end{example}

\begin{theorem}
\label{T4}Suppose that $f\,^{\prime \prime }(x)\geq 0$ on $\Re $ and that $%
\lim\limits_{x\rightarrow \,-\infty }\left( xf\,^{\prime \prime }(x)\right)
=A$, where $-\infty \leq A<0$ and $\lim\limits_{x\rightarrow \infty }\left(
xf\,^{\prime \prime }(x)\right) =A$, where $0<A\leq \infty $. Then $%
I_{f}(P)=2$ for any point $P=(s,t)$ below the graph of $f$.
\end{theorem}

\begin{proof}
The proof follows from Lemma \ref{L4} as in the proof of Theorem \ref{T3}
and we omit the details.
\end{proof}

\qquad We now apply Theorem \ref{T4} to show that any point below the graph
of a convex rational function or exponential polynomial must have
illumination index equal to $2$.

\begin{proposition}
\label{P1}\textbf{\ }Let $R$ be a rational function defined on $\Re $ with $%
R^{\prime \prime }\geq 0$ on $\Re $. Then $I_{R}(P)=2$ for any point $P$
below the graph of $R$.
\end{proposition}

\begin{proof}
Write $R(x)=\dfrac{p(x)}{q(x)}$, where $p$ and $q$ are polynomials of degree 
$m$ and $n$ respectively. If $n\geq m$, then $R$ has a horizontal asymptote
and thus $R^{\prime \prime }$ cannot be non--negative on $\Re $. Thus we
have $n<m$. If $m=n+1$, then $R$ has one oblique asymptote, $L$, which
implies that $\lim\limits_{x\rightarrow \pm \infty }(R(x)-L(x))=0$ and again 
$R^{\prime \prime }$ would not be non--negative on $\Re $. Thus $m-n-1\neq 0$%
. Also, since $R$ is defined on $\Re $, $n$ must be even. Let $%
p(x)=\sum\limits_{k=0}^{m}a_{k}x^{k}$, $q(x)=\sum%
\limits_{k=0}^{n}b_{k}x^{k},a_{m}\neq 0\neq b_{n}$. Then $R^{\prime }=\dfrac{%
qp^{\prime }-pq^{\prime }}{q^{2}}\Rightarrow $

$R^{\prime \prime }=\dfrac{q^{2}p^{\prime \prime }-pqq^{\prime \prime
}-2p^{\prime }q^{\prime }q+2p\left( q^{\prime }\right) ^{2}}{q^{3}}$, which
implies that 
\begin{gather*}
\left( \tsum\limits_{k=0}^{n}b_{k}x^{k}\right) ^{3}R^{\prime \prime }(x)= \\
\left( \tsum\limits_{k=0}^{n}b_{k}x^{k}\right) ^{2}\left(
\tsum\limits_{k=2}^{m}k(k-1)a_{k}x^{k-2}\right) - \\
\left( \tsum\limits_{k=0}^{m}a_{k}x^{k}\right) \left(
\tsum\limits_{k=0}^{n}b_{k}x^{k}\right) \left(
\tsum\limits_{k=2}^{n}k(k-1)b_{k}x^{k-2}\right) \\
-2\left( \tsum\limits_{k=1}^{m}ka_{k}x^{k-1}\right) \left(
\tsum\limits_{k=1}^{n}kb_{k}x^{k-1}\right) \left(
\tsum\limits_{k=0}^{n}b_{k}x^{k}\right) \\
+2\left( \tsum\limits_{k=0}^{m}a_{k}x^{k}\right) \left(
\tsum\limits_{k=1}^{n}kb_{k}x^{k-1}\right) ^{2}= \\
\left( m-n\right) \left( m-n-1\right) a_{m}b_{n}^{2}x^{2n+m-2}+\cdots .
\end{gather*}

If $2n+m-2\leq 3n$, then $m\leq n+2$, which implies that $m=n+2$ since $m>n$
and $m\neq n+1$. If $2n+m-2>3n$, then $2n+m-2-3n=\allowbreak m-n-2$ must be
even since $R^{\prime \prime }\geq 0$ as $\left\vert x\right\vert
\rightarrow \infty $. In either case, $m$ is also even. Since $m>n$ and $m$
is even, it follows that $m-n-1>0$. Since $R^{\prime \prime }\geq 0$ as $%
\left\vert x\right\vert \rightarrow \infty $, $\left( m-n\right) \left(
m-n-1\right) \dfrac{a_{m}}{b_{n}}>0$. Thus $\lim\limits_{x\rightarrow
-\infty }\left[ xR^{\prime \prime }(x)\right] =\lim\limits_{x\rightarrow
-\infty }\dfrac{\left( m-n\right) \left( m-n-1\right)
a_{m}b_{n}^{2}x^{2n+m-1}+\cdots }{b_{n}^{3}x^{3n}+\cdots }=$

$\lim\limits_{x\rightarrow -\infty }\dfrac{\left( m-n\right) \left(
m-n-1\right) a_{m}}{b_{n}}x^{\allowbreak m-n-1}=-\infty $ and

$\lim\limits_{x\rightarrow \infty }\left[ xR^{\prime \prime }(x)\right]
=\lim\limits_{x\rightarrow \infty }\dfrac{\left( m-n\right) \left(
m-n-1\right) a_{m}}{b_{n}}x^{\allowbreak m-n-1}=\infty $. By Theorem \ref{T4}%
, $I_{R}(P)=2$ for any point $P$ below the graph of $R$.
\end{proof}

\begin{proposition}
\label{P2}Suppose that $p$ and $q$ are polynomials of degree $m$ and $n$
respectively, and let $f(x)=p(x)e^{q(x)}$. Suppose that $f\,^{\prime \prime
}\geq 0$ on $\Re $. Then $I_{f}(P)=2$ for any point $P$ below the graph of $%
f $.
\end{proposition}

\begin{proof}
Let $p(x)=\sum\limits_{k=0}^{m}a_{k}x^{k}$, $q(x)=\sum%
\limits_{k=0}^{n}b_{k}x^{k},a_{m}\neq 0\neq $ $b_{n}$. Now we must have $%
b_{n}>0$ since if $b_{n}<0$, then $\lim\limits_{\left\vert x\right\vert
\rightarrow \infty }f(x)=0$. That would imply that $f\,^{\prime \prime
}\ngeq 0$ on $\Re $. A simple computation gives 
\begin{equation*}
f\,^{\prime \prime }=(p\left( q^{\prime }\right) ^{2}+2p^{\prime }q^{\prime
}+pq^{\prime \prime }+p^{\prime \prime })e^{q}
\end{equation*}%
and 
\begin{gather*}
p\left( q^{\prime }\right) ^{2}+2p^{\prime }q^{\prime }+pq^{\prime \prime
}+p^{\prime \prime }= \\
\left( \tsum\limits_{k=0}^{m}a_{k}x^{k}\right) \left(
\tsum\limits_{k=1}^{n}kb_{k}x^{k-1}\right) ^{2}+2\left(
\tsum\limits_{k=1}^{m}ka_{k}x^{k-1}\right) \left(
\tsum\limits_{k=1}^{n}kb_{k}x^{k-1}\right) + \\
\left( \tsum\limits_{k=0}^{m}a_{k}x^{k}\right) \left(
\tsum\limits_{k=2}^{n}k(k-1)b_{k}x^{k-2}\right) +\left(
\tsum\limits_{k=2}^{m}k(k-1)a_{k}x^{k-2}\right) =
\end{gather*}%
\begin{eqnarray*}
&&n^{2}a_{m}b_{n}^{2}x^{2n+m-2}+\cdots +2mna_{m}b_{n}x^{n+m-2}+\cdots + \\
&&n(n-1)a_{m}b_{n}x^{n+m-2}+\cdots +m(m-1)a_{m}x^{m-2}+\cdots 
\end{eqnarray*}%
$f\,^{\prime \prime }\geq 0$ on $\Re $ implies that $p\left( q^{\prime
}\right) ^{2}+2p^{\prime }q^{\prime }+pq^{\prime \prime }+p^{\prime \prime
}\geq 0$ on $\Re \Rightarrow 2n+m-2$ is even and $a_{m}>0$. Thus $2n+m-1$ is
odd and $\lim\limits_{x\rightarrow \,-\infty }\left[ xf\,^{\prime \prime }(x)%
\right] =n^{2}a_{m}b_{n}^{2}\lim\limits_{x\rightarrow \,-\infty
}x^{2n+m-1}=-\infty $. Similarly, $\lim\limits_{x\rightarrow \infty }\left[
xf\,^{\prime \prime }(x)\right] =\infty $. By Theorem \ref{T4}, $I_{f}(P)=2$
for any point $P$ below the graph of $f$.
\end{proof}

\section{Polynomials}

We now prove some results about the illumination index for polynomials with
real coefficients. As earlier, for given differentiable $f$, we let $%
g_{s}(c)=f(c)+(s-c)f\,^{\prime }(c)$. We also let $\pi _{n}=$ polynomials of
degree $\leq n$.

\begin{remark}
If one or more of the tangent lines which pass thru $P$ is a multiple
tangent line, then the illumination index of $P=(s,t)$ could be strictly
smaller than the number of real roots of $g_{s}(c)-t$. This will need to be
taken into account for some of the proofs below.
\end{remark}

The following lemma holds for more than just the polynomials, but we just
consider that case in this section.

\begin{lemma}
\label{L5}Let $f$ be a polynomial and let $s\in \Re $ be given. Then the
local extrema of $g_{s}(c)$ occur at precisely the following values of $c$.

(i) $c=s$ if $\left( s,f(s)\right) $ is \textbf{not }an inflection point of $%
f$

(ii) $c=d\neq s$ if $\left( d,f(d)\right) $ is an inflection point of $f$
\end{lemma}

\begin{proof}
It is easy to show that 
\begin{equation*}
g_{s}^{(k)}(c)=(s-c)f^{(k+1)}(c)-(k-1)f^{(k)}(c),k\geq 1.
\end{equation*}%
To prove (i): First, suppose that $f\,^{\prime \prime }(s)\neq 0$. Then $%
g_{s}(s)$ is a local extremum of $g_{s}(c)$ since then $g_{s}^{\prime }(s)=0$
and $g_{s}^{\prime \prime }(s)=-f\,^{\prime \prime }(s)\neq 0$. Now suppose
that $f^{(k)}(s)=0$ for $k=2,...,m-1$ and $f^{(m)}(s)\neq 0,m\geq 3$. Then $m
$ is even since $\left( s,f(s)\right) $ is not\textbf{\ }an inflection point
of $f$.  $g_{s}^{(k)}(s)=-(k-1)f^{(k)}(s)$, which implies that $%
g_{s}^{(k)}(s)=0$ for $k=2,...,m-1$ and $g_{s}^{(m)}(s)\neq 0$. Hence $%
g_{s}(s)$ is a local extremum of $g_{s}(c)$ since $m$ is even. That proves
(i).

To prove (ii): Suppose that $\left( d,f(d)\right) $ is an inflection point
of $f,d\neq s$. Suppose that $f^{(k)}(d)=0$ for $k=2,...,m-1$ and $%
f^{(m)}(d)\neq 0$, $m\geq 3$. Then $m$ is odd since $\left( d,f(d)\right) $
is an inflection point of $f$. Since $g_{s}^{(k)}(d)=0$ for $k=1,...,m-2$
and $g_{s}^{(m-1)}(d)=(s-d)f^{(m)}(d)-(k-1)f^{(m-1)}(d)=(s-d)f^{(m)}(d)\neq 0
$, $g_{s}(d)$\ is a local extremum of $g_{s}(c)$ since $m-1$ is even. That
proves (ii).
\end{proof}

Suppose that $f(x)=\tsum\limits_{k=0}^{n}a_{k}x^{k}$. Then a simple
computation yields 
\begin{equation}
g_{s}(c)-t=-(n-1)a_{n}c^{n}+\tsum\limits_{k=1}^{n-1}\left[
s(k+1)a_{k+1}-(k-1)a_{k}\right] c^{k}+a_{0}+sa_{1}-t,n\geq 2.  \label{6}
\end{equation}

\begin{remark}
By (\ref{6}) it follows immediately that if $n\geq 3$ is \textit{odd}, then $%
I_{f}(P)\geq 1$ for any $P\in G_{f}^{C}$.
\end{remark}

\begin{lemma}
\label{L6}Let $f(x)=\sum\limits_{k=0}^{n}a_{k}x^{k},a_{n}\neq 0$ and let $%
s\in \Re $.

(i) If\ $n\geq 3$ and odd, then

$\lim\limits_{c\rightarrow -\infty }g_{s}(c)=\left[ sgn(a_{n})\right] \infty 
$ and $\lim\limits_{c\rightarrow \infty }g_{s}(c)=\left[ -sgn(a_{n})\right]
\infty $

(ii) If\ $n~$is even, then

$\lim\limits_{c\rightarrow -\infty }g_{s}(c)=\left[ -sgn(a_{n})\right]
\infty $ and $\lim\limits_{c\rightarrow \infty }g_{s}(c)=\left[ -sgn(a_{n})%
\right] \infty $
\end{lemma}

\begin{proof}
The proof follows immediately from (\ref{6}).
\end{proof}

Our first theorem in this section is about cubic polynomials. We shall prove
some more results below for the case when $n$ is odd.

\begin{theorem}
\label{T5}Let $f$ be a cubic polynomial. Then for any $k=1,2,3$ there exists 
$P\in G_{f}^{C}$ such that $I_{f}(P)=k$.
\end{theorem}

\begin{proof}
Suppose, without loss of generality, that $f(x)=\tsum%
\limits_{k=0}^{3}a_{k}x^{k}$ with $a_{3}>0$. Then $f$ has exactly one
inflection point, $\left( d,f(d)\right) $. Choose any $s\neq d$. Then $%
g_{s}(c)$ has two exactly local extrema, $g_{s}(c_{1})$ and $g_{s}(c_{2})$,
by Lemma \ref{L5}, part (ii). Since $\lim\limits_{c\rightarrow -\infty
}g_{s}(c)=\infty $ and $\lim\limits_{c\rightarrow \infty }g_{s}(c)=-\infty $
by Lemma \ref{L6}, we may assume that $g_{s}(c_{1})$ equals the local
maximum and $g_{s}(c_{2})$ equals the local minimum, with $c_{1}<c_{2}$. If $%
t<g(c_{2})$ or $t>g(c_{1})$, then the horizontal line $y=t$ intersects the
graph of $g_{s}$ in one point, which implies that $I_{f}(P)=1$. If $%
g(c_{2})<t<g(c_{1})$, then the horizontal line $y=t$ intersects the graph of 
$g_{s}$ in three points, $\left( d_{i},g_{s}\left( d_{i}\right) \right)
,i=1,2,$ or $3$. Then $I_{f}(P)=3$ since a cubic polynomial cannot have any
multiple tangent lines. Finally, the horizontal lines $y=g(c_{1})$ and $%
y=g(c_{2})$ intersect the graph of $g_{s}$ in two points, which yields a
point, $P=(s,t)$, such that $I_{f}(P)=2$. One could also use the fact that $%
g_{d}(c)$ has one local extremum(by Lemma \ref{L5}, part (i)) to obtain a
point, $P=(s,t)$, such that $I_{f}(P)=2$.
\end{proof}

The following example shows that Theorem \ref{T5} does not hold in general
for $n$ odd, $n\geq 5$.

\begin{example}
Let $f(x)=x^{n},n\geq 5$ and odd. Then for any $s\in \Re $, $%
g_{s}(c)-t=c^{n}+(s-c)nc^{n-1}-t=-(n-1)c^{n}+nsc^{n-1}-t$ and $%
g_{s}(-c)-t=(n-1)c^{n}+nsc^{n-1}-t$. We consider the following six cases.

Case 1: $s,t>0$. Then $g_{s}(c)-t$ has $2$ sign changes and $g_{s}(-c)-t$
has $1$ sign change, which implies that $g_{s}(c)-t$\ has at most $3$\
distinct real roots.

Case 2: $s>0,t<0$. Then $g_{s}(c)-t$ has $1$ sign change and $g_{s}(-c)-t$
has $0$ sign changes, which implies that $g_{s}(c)-t$\ has at most $1$\ real
root.

Case 3: $s<0,t>0$. Then $g_{s}(c)-t$ has $0$ sign changes and $g_{s}(-c)-t$
has $1$ sign change, which implies that $g_{s}(c)-t$\ has at most $1$\ real
root.

Case 4: $s,t<0$. Then $g_{s}(c)-t$ has $1$ sign change and $g_{s}(-c)-t$ has 
$2$ sign changes, which implies that $g_{s}(c)-t$\ has at most $3$\ distinct
real roots.

Case 5: $s=0$. Then $g_{s}(c)-t=-(n-1)c^{n}-t$, which has $1$\ real root.

Case 6: $t=0$. Then $g_{s}(c)-t=-(n-1)c^{n}+nsc^{n-1}=c^{n-1}\left[
-(n-1)c+ns\right] $ has $2$\ distinct real roots.

Hence for any point $P\in G_{f}^{C}$, $I_{f}(P)\leq 3$
\end{example}

The example above shows that there are odd polynomials of any degree such
that $I_{f}(P)\leq 3$ for all $P\in G_{f}^{C}$. Our next result is a
positive result about the illumination index of all odd polynomials.

\begin{theorem}
Suppose that $f$ is a polynomial of degree\ $n\geq 5,n$ odd. Then there
exists $P\in G_{f}^{C}$ such that $I_{f}(P)=3$.
\end{theorem}

\begin{proof}
Since $n$ is odd, $f\,^{\prime \prime }$ must have at least one real root
where it changes sign. Hence $f$ has at least one inflection point, $\left(
d,f(d)\right) $. Choose any $s\neq d$. Arguing as in the proof of Theorem %
\ref{T5} above,$g_{s}(c)$ has two exactly local extrema, $g_{s}(c_{1})$ and $%
g_{s}(c_{2})$, by Lemma \ref{L5}, part (ii), and we may assume that $%
g_{s}(c_{1})$ equals the local maximum and $g_{s}(c_{2})$ equals the local
minimum, with $c_{1}<c_{2}$. However, it is possible that $f$ has multiple
tangent lines. Suppose that for each $t,g(c_{2})<t<g(c_{1})$, $y=t$
intersects the graph of $g_{s}$ in the three distinct points $\left(
d_{i},g_{s}\left( d_{i}\right) \right) $, $i=1,2,3$, and the tangent lines
at $\left( d_{i},g_{s}\left( d_{i}\right) \right) $ and at $\left(
d_{j},g_{s}\left( d_{j}\right) \right) $ are identical for some $i\neq j$.
Then $f$ would have infinitely many multiple tangent lines since the $\left(
d_{i},g_{s}\left( d_{i}\right) \right) $ change with $t$. But a polynomial
can only have finitely many multiple tangent lines(that is not difficult to
prove), and thus we can choose $t,g(c_{2})<t<g(c_{1})$, such that $\left(
d_{i},g_{s}\left( d_{i}\right) \right) $ yield three distinct tangent lines
to $f$, which yields $I_{f}(P)=3$.
\end{proof}

\begin{theorem}
Suppose that $f$ is a polynomial of degree\ $n\geq 2,n$ even. Then there
exist points $P_{1},P_{2}\in G_{f}^{C}$ such that $I_{f}(P_{1})=0$ and $%
I_{f}(P_{1})=2$.
\end{theorem}

\begin{proof}
Since $n$ is even, for any $s\in \Re $, $\lim\limits_{c\rightarrow -\infty
}g_{s}(c)=-\infty $ and $\lim\limits_{c\rightarrow \infty }g_{s}(c)=-\infty $%
, or $\lim\limits_{c\rightarrow -\infty }g_{s}(c)=\infty $ and $%
\lim\limits_{c\rightarrow \infty }g_{s}(c)=\infty $ by Lemma \ref{L6}. Thus
there must be values of $t$ such that the horizontal line $y=t$ does not
intersect the graph of $g_{s}$, which implies that $I_{f}(P)=0$ for $P=(s,t)$%
. Now suppose that $f\,^{\prime \prime }\geq 0$ on $\Re $. Then $I_{f}(P)=2$
for any point $P$ below the graph of $f$ by Proposition \ref{P1}. If $%
f\,^{\prime \prime }\ngeq 0$ on $\Re $, then $f$ has at least one inflection
point, $\left( d,f(d)\right) $, which implies that $g_{s}$ has at least one
local extremum. Again, since $\lim\limits_{c\rightarrow -\infty
}g_{s}(c)=-\infty $ and $\lim\limits_{c\rightarrow \infty }g_{s}(c)=-\infty $%
, or $\lim\limits_{c\rightarrow -\infty }g_{s}(c)=\infty $ and $%
\lim\limits_{c\rightarrow \infty }g_{s}(c)=\infty $, there must be values of 
$t$ such that the horizontal line $y=t$ intersects the graph of $g_{s}$ in
precisely two points,$\left( d_{i},g_{s}\left( d_{i}\right) \right) $, $i=1,2
$. In addition, one can choose $t$ so that the tangent lines to the graph of 
$f$ at $\left( d_{i},f\left( d_{i}\right) \right) $, $i=1,2$ are distinct,
which yields $I_{f}(P)=2$.
\end{proof}

\section{Theta Illumination Index \label{S3}}

Let $\theta $ be a given angle with $0\leq \theta \leq \dfrac{\pi }{2}$. We
call $L_{\theta }$ a \textbf{theta line }at\textbf{\ }$P$ if $L_{\theta }$
makes an angle, $\theta $, with the graph of $f$ at $P$. Here the angle
between two lines in a plane is defined to be $0$, if the lines are
parallel, or the smaller angle having as sides the half-lines starting from
the intersection point of the lines and lying on those two lines, if the
lines are not parallel.

\begin{definition}
Let $f(x)$ be a differentiable function on the real line and let $P\in
G_{f}^{C}$. Let $\theta $ be a given angle with $0\leq \theta \leq \dfrac{%
\pi }{2}$. We say that the $\theta $ \textit{illumination index} of $P$,
denoted by $I_{f,\theta }(P)$, is $k$ if there are $k$ distinct theta lines
to the graph of $f$ which pass through $P$. In particular, we use $%
I_{f,N}(P) $ to denote the $\dfrac{\pi }{2}$ \textit{illumination index} of $%
P$. In that case, of course, $L_{\theta }$ is a normal line to the graph of $%
f$.
\end{definition}

Unlike the case with illumination by tangent lines, where it is clearly
possible that \textbf{no} tangent line passes thru a given point $P\in
G_{f}^{C}$, this cannot happen with normal lines.

\begin{theorem}
\label{normal}For any differentiable $f$ defined on $\Re $ and any point $%
P\in G_{f}^{C}$, $I_{f,N}(P)\geq 1$
\end{theorem}

\begin{proof}
Given $P=(s,t)$, let $S=$ set of circles centered at $P$ which also
intersect $G_{f}$, and let $C_{0}$ be the circle in $S$ with the smallest
radius. Then $C_{0}$ is tangent to $G_{f}$ at some point $(c_{0},f(c_{0}))$.
The line $\overleftrightarrow{(s,t)\,(c_{0},f(c_{0}))}$ is then a normal
line passing thru $P$.
\end{proof}

\begin{remark}
Alternatively, one could also look at the distance from $P$ to $G_{f}$--the
minimal distance is obtained at $(c_{0},f(c_{0}))$ with $\overleftrightarrow{%
P\,(c_{0},f(c_{0}))}$ $\ $perpendicular to the tangent at $(c_{0},f(c_{0}))$.
\end{remark}

We believe that Theorem \ref{normal} holds for theta lines in general, but
have not been able to prove it.

\begin{conjecture}
For any differentiable $f$ defined on $\Re $, any given $0\leq \theta \leq 
\dfrac{\pi }{2}$, and any point $P\in G_{f}^{C}$, $I_{f,\theta }(P)\geq 1$.
\end{conjecture}

\end{document}